\documentclass{article}

\usepackage{amsmath, amssymb, amsthm}
\usepackage{hyperref}
\usepackage{cleveref}
\usepackage{chngcntr}
\usepackage{thmtools}
\usepackage{thm-restate}
\usepackage{fancyhdr}
\usepackage{dsfont}
\usepackage[english]{babel}
\usepackage[utf8]{inputenc}
\usepackage{commath}
\usepackage{marvosym}
\usepackage{mathtools}

\usepackage{thmtools}
\usepackage{thm-restate}
\usepackage{soul}

\usepackage{tikz-cd}

\usepackage[margin=1.5in]{geometry}

\usepackage{comment}

\newtheorem{theorem}{Theorem}[section]
\newtheorem{corollary}[theorem]{Corollary}
\newtheorem{proposition}[theorem]{Proposition}

\newtheorem{lemma}[theorem]{Lemma}

\newtheorem{observation}[theorem]{Observation}

\newtheorem*{theorem*}{Theorem}
\newtheorem*{corollary*}{Corollary}
\newtheorem*{lemma*}{Lemma}
\newtheorem*{proposition*}{Proposition}

\theoremstyle{definition}
\newtheorem{definition}[theorem]{Definition}
\newtheorem{remark}[theorem]{Remark}
\newtheorem{example}[theorem]{Example}

\newtheorem{setup}[theorem]{Setup}

\newtheorem{thmx}{Theorem}

\DeclareMathOperator{\tor}{Tor}
\DeclareMathOperator{\ext}{Ext}

\DeclareMathOperator{\ann}{ann}
\DeclareMathOperator{\depth}{depth}
\DeclareMathOperator{\Hom}{Hom}
\DeclareMathOperator{\supp}{Supp}
\DeclareMathOperator{\height}{ht}

\DeclareMathOperator{\idim}{id}
\DeclareMathOperator{\grade}{grade}

\DeclareMathOperator{\ass}{Ass}

\title{Bass Numbers of the First Nonzero Local Cohomology Module}
\author{Andrew J. Soto Levins}
\date{}

\begin{document}


\maketitle
\begin{abstract}
Let $R$ be a regular local ring containing a field, let $I$ be an ideal with $d=\height{I}$, and assume $\height{p}=d$ for every minimal prime $p$ of $I$. We compute the Bass numbers $\mu^{0}(q,H_{I}^{d}(R))$ and $\mu^{1}(q,H_{I}^{d}(R))$ for all primes $q$. We then study $\mu^{2}(q,H_{I}^{d}(R))$ by considering the associated primes of $H_{I}^{d+1}(R)$.
\end{abstract}

\section{Introduction}
Let $(R,m,k)$ be a Noetherian local ring and let $M$ be a module. In this setting, $M$ has a minimal injective resolution
$$E=0\rightarrow E^{0}\rightarrow E^{1}\rightarrow\dots$$
that is unique up to isomorphism. Since every injective module can be written uniquely (up to isomorphism) as a direct sum of indecomposable injective modules, the number of summands isomorphic to $E_{R}(R/q)$ for a given prime $q$ that appear in $E^{i}$ is well defined. We call this number the $ith$ number Bass of $M$ with respect to $q$ and denote it by $\mu^{i}(q,M)$. A natural question to ask is when is $\mu^{i}(q,M)$ finite? If $M$ is finitely generated, then $\mu^{i}(q,M)$ is finite for all $i$ and all $q$ (\cite[Theorem 18.7]{matsumura}). In this paper we are interested in Bass numbers of local cohomology. In \cite{hartshorne} Hartshorne gave an example of a local cohomology module $H_{I}^{i}(R)$ over a ring $R$ that is not regular with $\mu^{0}(m,H_{I}^{i}(R))$ infinite. Surprisingly, Huneke and Sharp in \cite[Theorem 2.1]{hunekeone} and Lyubeznik in \cite[Theorem 3.4.d]{lyubeznik} prove that for a regular local ring $R$ containing a field, $\mu^{i}(q,H_{I}^{j}(R))$ is finite for all $i$, all $j$, all ideals $I$ and all primes $q$. Results like this are unexpected since local cohomology modules are not finitely generated in general. The goal of this paper is to give a better understanding of $\mu^{0}(q,H_{I}^{d}(R))$, $\mu^{1}(q,H_{I}^{d}(R))$, and $\mu^{2}(q,H_{I}^{d}(R))$.\newline

Our main result is as follows.
\begin{thmx}[Main Theorem] Let $(R,m,k)$ be a regular local ring containing a field with $n=\dim{R}$, let $I$ be an ideal with $d=\height{I}$, and assume $\height{p}=d$ for every minimal prime $p$ of $I$. Let 
$$E=0\rightarrow E^{0}\rightarrow E^{1}\rightarrow\dots$$
be the minimal injective resolution of $H_{I}^{d}(R)$. Then the following hold.
\begin{enumerate}
\item(Lemma \ref{firsttwobassnumbersanddplusonelocalcoho}) For $i=0,1$ we have
\[\mu^{i}(q,H_{I}^{d}(R))= \begin{cases} 
      1 & q\supseteq I, \height{q}=d+i \text{ and} \\
      0 & \text{otherwise}. 
   \end{cases}
\]
In particular, $\ass{E^{i}}=\{q \ | \ q\supseteq I,\height{q}=d+i\}$ for $i=0,1$.
\item(Lemma \ref{firsttwobassnumbersanddplusonelocalcoho}) If $q\notin\ass{H_{I}^{d+1}(R)}$, $q\supseteq I$, and $\height{q}=d+2$, then $\mu^{2}(q,H_{I}^{d}(R))=1$. 
\item(Proposition \ref{bassnumbersequalityplusone}) If $n-d\geq 3$, $\height{q}=d+2$, and $q\in\ass{H_{I}^{d+1}(R)}$, then
$$\mu^{0}(q,H_{I}^{d+1}(R))+1=\mu^{2}(q,H_{I}^{d}(R)).$$
\end{enumerate}
\end{thmx}
The assumption that $R$ is a regular local ring containing a field can be weakened in the first two parts.\newline

In Lemma \ref{whenislocalcohocm} we consider primes $q$ with $\height{q}=\height{I}+2$ and show
$$\ass{E^{2}}\supseteq\{q \ | \ q\supseteq I, \height{q}=\height{I}+2\}.$$
A natural question to ask is when do we have
$$(1)\quad\quad\ass{E^{2}}=\{q \ | \ q\supseteq I, \height{q}=\height{I}+2\}.$$
Under the assumptions of the main theorem, if $d<n$, then $\dim{\supp{H_{I}^{d+1}(R)}}\leq n-d-2$ by Lemma \ref{upperboundfordimensionofsupportoflocalcohomology}. In particular, if $d\geq n-2$, then $\supp{H_{I}^{d+1}(R)}\subseteq \{m\}$. In the next case, when $n-d=3$, we give a characterization for when $(1)$ holds by using the associated primes and injective dimension of $H_{I}^{d+1}(R)$ (see Proposition \ref{nminusdequalsthreewithfourequivalentconditions}). In Section \ref{examples} we provide examples of local cohomology modules that satisfy some of the assumptions and conclusions that show up throughout the paper.\newline

I would like to thank Nawaj KC, Taylor Murray, and Mark Walker for conversations about local cohomology, and Jack Jeffries for such conversations, for his comments on an earlier draft of this paper, and for the suggestion to add Section \ref{examples}. Finally I thank Alexandra Seceleanu for a helpful discussion that led to Example \ref{alexandraexample} and Proposition \ref{alexandraproposition}.

\section{Lower Bound for Injective Dimension}
The goal of this section is to prove Lemma \ref{lemmatodropthethereisaninjectionassumption}, which plays an important role in proving the main theorem. Throughout the paper, we write CM for Cohen Macaulay and MCM for Maximal Cohen Macaulay. The following definition is taken from \cite[Definition 4.6.1]{roberts}.
\begin{definition} Let $R$ be a Noetherian local ring of dimension d. A dualizing complex for $R$ is a complex 
$$D=0\rightarrow D^{0}\rightarrow\dots\rightarrow D^{d}\rightarrow 0$$
so that
\begin{enumerate}
\item The homology of $D$ is finitely generated.
\item Each module $D^{i}$ is a sum of $E_{R}(R/p)$, each occurring exactly once, where $p$ is prime with $\dim{R/p}=d-i$.
\end{enumerate}
\end{definition}

\begin{lemma} \label{dualizingcomlexinjrescanmod} Let $(R,m,k)$ be a CM local ring with canonical module $\omega_{R}$. Then $R$ has a dualizing complex $D$ and $D$ is an injective resolution for $\omega_{R}$.
\end{lemma}

\begin{proof}
Note that Gorenstein local rings have a dualizing complex by \cite[Theorem 18.8]{matsumura} and \cite[Definition 4.6.1]{roberts}, and $R$ is a quotient of a Gorenstein local ring by \cite[Theorem 3.3.6]{brunsherzog}, so $R$ has a dualizing complex $D$ by \cite[Lemma 4.6.2]{roberts}. Since $R$ is CM, $D$ is an injective resolution of $H=H^{0}(D)$ by \cite[Proposition 4.6.3]{roberts}.\newline

To see that $H$ is the canonical module, note that $H$ has finite injective dimension since it is resolved by $D$. Also, we have the quasiisomorphism
$$\Hom_{R}(k,D)\simeq \Hom_{R}(k,E_{R}(k)),$$ 
so $\ext_{R}^{i}(k,H)=0$ for $i\neq\dim{R}$ and 
$$\ext_{R}^{\dim{R}}(k,H)\cong\Hom_{R}(k,E_{R}(k)),$$
telling us that the type of $H$ is one and $\depth{H}\geq\dim{R}$. Since $H$ is a finitely generated MCM module of finite injective dimension that also has type one, $H$ is a canonical module, so $\omega_{R}\cong H$ since the canonical module is unique up to isomorphism. 
\end{proof}

\begin{lemma} \label{noitorsion} Let $R$ be a Noetherian ring and let $I$ be an ideal. If $q$ is a prime ideal with $\height{q}<\height{I}$, then $E_{R}(R/q)$ has no $I$ torsion, meaning that $\Gamma_{I}(E_{R}(R/q))=0$. In particular, 
$$\Hom_{R}(H_{I}^{i}(\_),E_{R}(R/q))=0$$
for all $i$.
\end{lemma}

\begin{proof}
Choose a nonzero $x$ in $I\backslash q$. Then multiplication by $x$ induces an automorphism on $E_{R}(R/q)$ by \cite[Theorem 18.4.3]{matsumura}, telling us that no nonzero element of $E_{R}(R/q)$ can be killed by $I$.
\end{proof}

Taylor Murray taught us part 2 of the following lemma.
\begin{lemma} \label{dimsuppequalsnminusd} Let $R$ be a CM local ring with canonical module $\omega_{R}$ and let $I$ be an ideal. Then
\begin{enumerate}
\item $\depth_{I}R=\height{I}=\depth_{I}\omega_{R}$. In particular, $H_{I}^{\height{I}}(\omega_{R})\neq 0$.
\item $\dim{\supp{H_{I}^{\height{I}}(\omega_{R})}}=\dim{R}-\height{I}$ and $\height{p}=\height{I}$ for every associated prime $p$ of $H_{I}^{\height{I}}(\omega_{R})$. 
\end{enumerate}
\end{lemma}

\begin{proof}
For the first two equalities, note that $\omega_{R}$ is MCM, so every $R$ regular sequence is an $\omega_{R}$ regular sequence, so $\depth_{I}R\leq \depth_{I}\omega_{R}$. This gives us
$$\height{I}=\depth_{I}R\leq \depth_{I}\omega_{R}\leq\height{I}.$$
For the third equality, notice that if $D$ is the dualizing complex, then
$$\Gamma_{I}(D)=0\rightarrow\bigoplus_{p\supseteq I,\height{p}=\height{I}}E_{R}(R/p)\xrightarrow{f}\bigoplus_{p\supseteq I,\height{p}=\height{I}+1}E_{R}(R/p)\rightarrow\dots$$
by Lemma \ref{noitorsion}, so
$$H_{I}^{\height{I}}(\omega_{R})=H^{\height{I}}(\Gamma_{I}(D))=\text{ker}f\subseteq \bigoplus_{p\supseteq I,\height{p}=\height{I}}E_{R}(R/p).$$
where the first equality is by Lemma \ref{dualizingcomlexinjrescanmod}. Since all the associated primes of $\bigoplus_{p\supseteq I,\height{p}=\height{I}}E_{R}(R/p)$ have height equal to $\height{I}$, all the associated primes of $H_{I}^{\height{I}}(\omega_{R})$ have $\height{I}$, so all the minimal elements in $\supp{H_{I}^{\height{I}}(\omega_{R})}$ have $\height{I}$. Since $R$ is CM, $\dim{R/p}=\dim{R}-\height{p}$ for each minimal prime $p$ in $\supp{H_{I}^{\height{I}}(\omega_{R})}$, so if $p$ is a minimal prime of $H_{I}^{\height{I}}(\omega_{R})$ we have
$$\dim{\supp{H_{I}^{\height{I}}(\omega_{R})}}=\dim{R/p}=\dim{R}-\height{p}=\dim{R}-\height{I},$$
finishing the proof.
\end{proof}

\begin{definition} Let $R$ be a Noetherian local ring and let $M$ be a module. Write $M^{\vee}$ for the Matlis dual of $M$.
\end{definition}

Mark Walker taught us the following fact.
\begin{lemma} \label{nonzeromaplemma} Let $R$ be a Noetherian local ring and let $f:M\rightarrow N$ be a nonzero module homomorphism. Then $f^{\vee}:N^{\vee}\rightarrow M^{\vee}$ is a nonzero module homomorphism.
\end{lemma}

\begin{proof}
Suppose $f^{\vee}$ is zero so that $f^{\vee\vee}$ is zero. Consider the following commutative diagram

\[\begin{tikzcd}
	{M^{\vee\vee}} &&&& {N^{\vee\vee}} \\
	\\
	\\
	{M} &&&& {N}
	\arrow["{f^{\vee\vee}}", from=1-1, to=1-5]
	\arrow["f", from=4-1, to=4-5]
	\arrow["{\phi_{M}}"', from=4-1, to=1-1]
	\arrow["{\phi_{N}}"', from=4-5, to=1-5]
\end{tikzcd}\]
where the vertical maps are the natural map. The vertical maps are injective by \cite[Theorem 18.6.1]{matsumura}, and so nonzero. Notice that we now have
$$0\neq \phi_{N}f=f^{\vee\vee}\phi_{M}=0,$$
a contradiction.
\end{proof}

\begin{lemma} \label{injectionlemmaone} Let $(R,m,k)$ be a complete CM local ring with canonical module $\omega_{R}$ and $n=\dim{R}$. Let $I$ be an ideal with $d=\height{I}$. If there exists a nonzero map
$$\omega_{R}\rightarrow\Hom_{R}(H_{I}^{d}(\omega_{R}),H_{I}^{d}(\omega_{R})),$$
then $H_{m}^{n-d}(H_{I}^{d}(\omega_{R}))\neq 0$. In particular, we have
$$\idim_{R}H_{I}^{d}(\omega_{R})\geq n-d=\dim{\supp{H_{I}^{d}(\omega_{R})}}.$$
\end{lemma}

\begin{proof}
By Lemma \ref{dimsuppequalsnminusd} we have $n-d=\dim{\supp{H_{I}^{d}(\omega_{R})}}$. Also, if $H_{m}^{n-d}(H_{I}^{d}(\omega_{R}))$ is not Artinian, then $H_{m}^{n-d}(H_{I}^{d}(\omega_{R}))\neq 0$ since zero is Artinian, so we can assume $H_{m}^{n-d}(H_{I}^{d}(\omega_{R}))$ is Artinian.\newline

Let $\underline{x}$ be a maximal $R$ regular sequence and let $C$ be the Cech complex on $\underline{x}$. Since $\sqrt{(\underline{x})}=m$, we have $H_{(\underline{x})}^{i}(\_)\cong H_{m}^{i}(\_)$. Since $R$ is CM, $C$ is a flat resolution of $H_{m}^{n}(R)$, so
$$H_{m}^{i}(H_{I}^{d}(\omega_{R}))\cong H^{i}(C\otimes_{R}H_{I}^{d}(\omega_{R}))\cong \tor_{n-i}(H_{I}^{d}(\omega_{R}),H_{m}^{n}(R)).$$
Applying Matlis duality to both sides with $i=n-d$, we see
\begin{align*}
(2)\quad\quad H_{m}^{n-d}(H_{I}^{d}(\omega_{R}))^{\vee}\cong \tor_{d}^{R}(H_{I}^{d}(\omega_{R}),H_{m}^{n}(R))^{\vee}&\cong \ext_{R}^{d}(H_{I}^{d}(\omega_{R}),H_{m}^{n}(R)^{\vee}) \\
&\cong \ext_{R}^{d}(H_{I}^{d}(\omega_{R}),\omega_{R}),
\end{align*}
where $H_{m}^{n}(R)^{\vee}\cong \omega_{R}$ by \cite[Theorem 3.5.8]{brunsherzog} since $R$ is complete and CM. Since $R$ is CM, $d=\depth_{I}R$, so
$$\ext_{R}^{d}(H_{I}^{d}(\omega_{R}),\omega_{R})\cong \Hom_{R}(H_{I}^{d}(\omega_{R}),H_{I}^{d}(\omega_{R}))$$
by \cite[Proposition 2.2]{mahmood}. Since there exists an a nonzero map
$$\omega_{R}\rightarrow\Hom_{R}(H_{I}^{d}(\omega_{R}),H_{I}^{d}(\omega_{R})),$$
there exists a nonzero map 
$$\ext_{R}^{d}(H_{I}^{d}(\omega_{R}),\omega_{R})^{\vee}\rightarrow \omega_{R}^{\vee}\cong H_{m}^{n}(R)$$
by Lemma \ref{nonzeromaplemma}. Since $R$ is complete and since $H_{m}^{n-d}(H_{I}^{d}(\omega_{R}))$ is Artinian, $H_{m}^{n-d}(H_{I}^{d}(\omega_{R}))\cong H_{m}^{n-d}(H_{I}^{d}(\omega_{R}))^{\vee\vee}$, so $(2)$ gives
$$H_{m}^{n-d}(H_{I}^{d}(\omega_{R}))\cong \ext_{R}^{d}(H_{I}^{d}(\omega_{R}),\omega_{R})^{\vee}.$$
Therefore there exists a nonzero map $H_{m}^{n-d}(H_{I}^{d}(\omega_{R}))\rightarrow H_{m}^{n}(R)$, finishing the proof.
\end{proof}

\begin{lemma} \label{injectionlemmatwo} Let $R$ be a ring and let $M$ be a nonzero module. Then there exists a nonzero map
$$R\rightarrow \Hom_{R}(M,M).$$
\end{lemma}

\begin{proof}
Define $f:R\rightarrow \Hom_{R}(M,M)$ by $r\mapsto \phi_{r}:M\rightarrow M$ where $\phi_{r}(x)=rx$. If $f(r)=0$ for all $r$, then $\phi_{r}=0$ for all $r$, telling us that $\ann{M}=R$, a contradiction.
\end{proof}

Recall that a module $M$ is faithful if $\ann{M}=0$.
\begin{lemma} \label{gradeandfaithfullemma} Let $R$ be a Noetherian local ring and let $M$ be a finitely generated faithful module. Then $\Hom_{R}(M,R)\neq 0$.
\end{lemma}

\begin{proof}
By definition, the grade of $M$ is the longest regular sequence in $\ann{M}$, so $\grade{M}=0$ since $M$ is faithful. By \cite[Proposition 1.2.10.e and Definition 1.2.11]{brunsherzog}, we have
$$\grade{M}=\inf\{n|\ext_{R}^{n}(M,R)\neq 0\},$$
so $\Hom_{R}(M,R)\neq 0$
\end{proof}

\begin{lemma} \label{lemmafromfathi} Let $R$ be a CM local domain with canonical module $\omega_{R}$ and let $I$ be an ideal with $d=\height{I}$. Then there exists an injection
$$R\rightarrow\Hom_{R}(H_{I}^{d}(\omega_{R}),H_{I}^{d}(\omega_{R})).$$
\end{lemma}

\begin{proof}
Consider the map
$$f:R\rightarrow\Hom_{R}(H_{I}^{d}(\omega_{R}),H_{I}^{d}(\omega_{R}))$$
defined by $r\mapsto\phi_{r}:H_{I}^{d}(\omega_{R})\rightarrow H_{I}^{d}(\omega_{R})$ where $\phi_{r}(x)=rx$. To show $f$ is one to one, it is enough to show $\ann{H_{I}^{d}(\omega_{R})}=0$. Note that since $R$ is a domain, the kernel of the natural map $\omega_{R}\rightarrow(\omega_{R})_{0}$ is zero, so \cite[Corollary 5.5]{fathi} gives the first equality below and \cite[Proposition 3.3.11.c]{brunsherzog} gives the second equality below
$$\ann{H_{I}^{d}(\omega_{R})}=\ann{\omega_{R}}=0.$$
\end{proof}

\begin{lemma} \label{lemmatodropthethereisaninjectionassumption} Let $(R,m,k)$ be a complete CM local ring with canonical module $\omega_{R}$ and $n=\dim{R}$. Let $I$ be an ideal with $d=\height{I}$. Assume $R$ is Gorenstein or assume $R$ is a domain. Then $H_{m}^{n-d}(H_{I}^{d}(\omega_{R}))\neq 0$. In particular, we have
$$\idim_{R}H_{I}^{d}(\omega_{R})\geq n-d=\dim{\supp{H_{I}^{d}(\omega_{R})}}.$$
\end{lemma}

\begin{proof}
Note that $H_{I}^{d}(\omega_{R})$ is nonzero by Lemma \ref{dimsuppequalsnminusd}. If $R$ is Gorenstein, then there exists a nonzero map
$$\omega_{R}\cong R\rightarrow \Hom_{R}(H_{I}^{d}(\omega_{R}),H_{I}^{d}(\omega_{R}))$$
by Lemma \ref{injectionlemmatwo}, and so $H_{m}^{n-d}(H_{I}^{d}(\omega_{R}))\neq 0$ by Lemma \ref{injectionlemmaone}.\newline

Now assume $R$ is a domain. Since $\omega_{R}$ is faithful by \cite[Proposition 3.3.11.c]{brunsherzog}, there exists a nonzero map $\omega_{R}\rightarrow R$ by Lemma \ref{gradeandfaithfullemma}. Also, by Lemma \ref{lemmafromfathi} there exists an injection
$$R\rightarrow\Hom_{R}(H_{I}^{d}(\omega_{R}),H_{I}^{d}(\omega_{R})),$$
which gives a nonzero map
$$\omega_{R}\rightarrow R\rightarrow\Hom_{R}(H_{I}^{d}(\omega_{R}),H_{I}^{d}(\omega_{R})),$$
and so $H_{m}^{n-d}(H_{I}^{d}(\omega_{R}))\neq 0$ by Lemma \ref{injectionlemmaone}.
\end{proof}

The following definition is taken from \cite[Definition 3.3.16]{brunsherzog}.
\begin{definition} Let $R$ be a CM ring. A finitely generated module $\omega_{R}$ is a canonical module of $R$ if $(\omega_{R})_{m}$ is a canonical module of $R_{m}$ for all maximal ideals $m$ of $R$.
\end{definition}

Note that since canonical modules localize by \cite[Theorem 3.3.5.b]{brunsherzog}, if $\omega_{R}$ is a canonical module for a CM ring $R$, then $(\omega_{R})_{p}$ is a canonical module for $R_{p}$ for all primes $p$. Also, In \cite[Proposition 3.5]{dorreh}, Dorreh proved that if $R$ is a regular local ring of dimension $n$ containing a field and $I$ is an ideal with $d=\height{I}$, then 
$$\idim_{R}H_{I}^{d}(R)=\dim{\supp{H_{I}^{d}(\omega_{R})}}=n-d.$$ 
\begin{proposition} \label{dorrehgeneralization} Let $R$ be a CM ring with canonical module $\omega_{R}$, let $I$ be an ideal, and let $p\supseteq I$ be a prime ideal with $\height{I}=\height{I_{p}}$. Assume $R$ is Gorenstein or assume $\widehat{R_{p}}$ is a domain. Then $H_{I}^{\height{I}}(\omega_{R})\neq 0$ and
$$\idim_{R}H_{I}^{\height{I}}(\omega_{R})\geq \height{p}-\height{I}.$$
\end{proposition}

\begin{proof}
Since $\idim_{R}H_{I}^{\height{I}}(\omega_{R})\geq \idim_{R_{p}}H_{I}^{\height{I}}(\omega_{R})_{p}$, since $\height{I}=\height{I_{p}}$, and since $H_{I}^{\height{I}}(\omega_{R})_{p}\cong H_{I_{p}}^{\height{I_{p}}}(\omega_{R_{p}})$, it is enough to show 
$$\idim_{R_{p}}H_{I_{p}}^{\height{I_{p}}}(\omega_{R_{p}})\geq \dim{R_{p}}-\height{I_{p}},$$
but this is immediate from the following three observations.
\begin{enumerate}
\item $\height{I_{p}\widehat{R_{p}}}=\dim{\widehat{R_{p}}}-\dim{\widehat{R_{p}}/I_{p}\widehat{R_{p}}}=\dim{R_{p}}-\dim{R_{p}/I_{p}}=\height{I_{p}}$
\item $\widehat{R_{p}}$ is a faithfully flat extension of $R_{p}$ so an $R_{p}$ module $M$ is nonzero if and only if $M\otimes_{R_{p}}\widehat{R_{p}}$ is nonzero
\item $H_{p_{p}}^{\dim{R_{p}}-\height{I_{p}}}(H_{I_{p}}^{\height{I_{p}}}(\omega_{R_{p}}))\otimes_{R_{p}}\widehat{R_{p}}$ is nonzero by Lemma \ref{lemmatodropthethereisaninjectionassumption}.\qedhere
\end{enumerate}
\end{proof}

\begin{corollary} \label{injectivedimlowerbound} Let $R$ be a CM ring with $n=\dim{R}$ and canonical module $\omega_{R}$ and let $I$ be an ideal with $d=\height{I}$. Assume all the maximal ideals of $R$ have the same height. Assume $R$ is Gorenstein or assume $\widehat{R_{m}}$ is a domain for all maximal ideals $m$. Then $H_{I}^{\height{I}}(\omega_{R})\neq 0$ and $\idim_{R}H_{I}^{\height{I}}(\omega_{R})\geq n-d$. In particular, if $n-d>0$, then $H_{I}^{d}(\omega_{R})$ is not injective.
\end{corollary}

\begin{proof}
Choose a prime ideal $p$ so that $\height{p}=d$ and choose a maximal ideal $m$ containing $p$. Then $\height{I}=\height{I_{m}}$ since if $p_{0}\subseteq\dots\subseteq p_{d}=p\subseteq\dots\subseteq p_{n}=m$ is a chain in $R$, then it is a chain in $R_{m}$. Therefore
$$\idim_{R}H_{I}^{\height{I}}(\omega_{R})\geq \height{m}-\height{I}=n-d$$
where the inequality is by Proposition \ref{dorrehgeneralization}.
\end{proof}

\section{Main Theorem} \label{maintheorem}
The goal of this section is to prove the main theorem. We do this by making a comparison between the minimal injective resolution of $H_{I}^{d}(\omega_{R})$ and the dualizing complex. The majority of the work goes into proving Lemma \ref{firsttwobassnumbersanddplusonelocalcoho}.

\begin{remark} \label{dimrminushtigreaterthantwojustification} Let $(R,m,k)$ be a CM local domain with canonical module $\omega_{R}$ and $n=\dim{R}$ and let $I$ be an ideal with $d=\height{I}$. Assume $R$ is Gorenstein or assume $\widehat{R}$ is a domain. If $n-d\leq 1$, then $\idim_{R}H_{I}^{d}(\omega_{R})=n-d$ and $H_{I}^{d}(\omega_{R})$ is the only nonzero local cohomology module. If $n-d=2$, $H_{I}^{i}(\omega_{R})=0$ for $i\neq d,d+1$.
\end{remark}

\begin{proof}
If $n-d=0$, then $\sqrt{I}=m$, so $H_{I}^{d}(\omega_{R})\cong E_{R}(k)$ by \cite[Corollary 3.5.9]{brunsherzog}. If $n-d=1$, then $\height{I}=d=n-1<\dim{R}$. Since $\sqrt{0+I}\neq m$, \cite[Theorem 1.1]{bagheriyeh} gives $H_{I}^{n}(R)=0$. Since the cohomological dimension of $M$ is less than or equal to the cohomological dimension of $R$ for any module $M$, we have $H_{I}^{n}(\omega_{R})=0$. Therefore applying $\Gamma_{I}$ to the dualizing complex of $R$, we get an exact sequence
$$0\rightarrow H_{I}^{n}(\omega_{R})\rightarrow \bigoplus_{p\supseteq I, \height{p}=d}E_{R}(R/p)\rightarrow E_{R}(k)\rightarrow 0,$$
and so $\idim_{R}H_{I}^{n}(\omega_{R})\leq 1=n-d$. Also, $\idim_{R}H_{I}^{n}(\omega_{R})\geq n-d=1$ by Corollary \ref{injectivedimlowerbound}, so $\idim_{R}H_{I}^{n}(\omega_{R})=1=n-d$.\newline

Again using \cite[Theorem 1.1]{bagheriyeh}, if $n-d=2$, since $\sqrt{0+I}\neq m$, $H_{I}^{d+2}(R)=0$. This tells us that $H_{I}^{i}(\omega_{R})=0$ for $i<d$ and $i>d+1$.
\end{proof}

\begin{lemma} \label{dualizingcomplexcomputationlemma} Let $R$ be a CM local ring with canonical module $\omega_{R}$ and $n=\dim{R}$ and let $I$ be an ideal with $d=\height{I}$. If $E$ is the minimal injective resolution of $H_{I}^{d}(\omega_{R})$, then the following hold.
\begin{enumerate}
\item $\ass{E^{0}}\subseteq\{q \ | \ q\supseteq I,\height{q}=d\}$
\item $\ass{E^{1}}\subseteq\{q \ | \ q\supseteq I,\height{q}=d+1\}$
\end{enumerate}
\end{lemma}

\begin{proof}
If $D$ is the dualizing complex, then
$$\Gamma_{I}(D)=0\rightarrow\bigoplus_{p\supseteq I,\height{p}=d}E_{R}(R/p)\xrightarrow{f}\bigoplus_{p\supseteq I,\height{p}=d+1}E_{R}(R/p)\rightarrow\dots$$
and
$$H_{I}^{\height{I}}(\omega_{R})=H^{\height{I}}(\Gamma_{I}(D))=\text{ker}f\subseteq \bigoplus_{p\supseteq I,\height{p}=d}E_{R}(R/p).$$
We now have a diagram with exact rows

\[\begin{tikzcd}
	0 & {H_{I}^{d}(\omega_{R})} & {E^{0}} & {\text{coker}(i)} & 0 \\
	\\
	\\
	0 & {H_{I}^{d}(\omega_{R})} & {\bigoplus E_{R}(R/p)} & {\text{coker}(i')} & 0
	\arrow[from=4-1, to=4-2]
	\arrow[from=4-2, to=4-3]
	\arrow[from=4-3, to=4-4]
	\arrow[from=4-4, to=4-5]
	\arrow[from=1-1, to=1-2]
	\arrow[from=1-2, to=1-3]
	\arrow[from=1-3, to=1-4]
	\arrow[from=1-4, to=1-5]
	\arrow[from=1-2, to=4-2]
	\arrow[dashed, from=1-3, to=4-3]
	\arrow[dashed, from=1-4, to=4-4]
\end{tikzcd}\]
where $i$ and $i'$ are the inclusions and the left vertical map is the identity. Since $\oplus E_{R}(R/p)$ is injective, we get the middle vertical map. Since $E^{0}$ is an essential extension of $H_{I}^{d}(\omega_{R})$, the middle vertical map is one to one, and the middle vertical map induces the right vertical map. Now the snake lemma tells us that the right vertical map is one to one. Therefore
$$\text{coker}(i)\subseteq \text{coker}(i')\subseteq D^{d+1},$$
so $\ass\text{coker}(i)\subseteq \ass{D^{d+1}}$. Since $E^{1}$ is the injective hull of $\text{coker}(i)$, $\ass{E^{1}}=\ass\text{coker}(i)$. This finishes the proof.
\end{proof}

\begin{lemma} \label{nminusdcompositionisnonzero} Let $(R,m,k)$ be a CM local ring with canonical module $\omega_{R}$ and $n=\dim{R}$ and let $I$ be an ideal with $d=\height{I}$. Assume $R$ is Gorenstein or assume $\widehat{R}$ is a domain. Then $H_{m}^{n-d}(H_{I}^{d}(\omega_{R}))\neq 0$.
\end{lemma}

\begin{proof}
Note that $\widehat{R}$ is faithfully flat and
$$H_{m}^{n-d}(H_{I}^{d}(\omega_{R}))\otimes_{R}\widehat{R}\cong H_{\widehat{m}}^{n-d}(H_{\widehat{I}}^{d}(\omega_{\widehat{R}}))\neq 0$$
by Lemma \ref{lemmatodropthethereisaninjectionassumption}.
\end{proof}

The following is contained in \cite[Theorem 6.1 and Theorem 6.2]{sri}.
\begin{theorem}[Iyengar] Let $R$ be a Noetherian ring, let $I$ be an ideal, and let $M$ be a module. Then
$$\depth(I,M)=\inf\{i|\ext_{R}^{i}(R/I,M)\neq 0\}=\inf\{i|H_{I}^{i}(M)\neq 0\}.$$
\end{theorem}

Recall that if $M$ is a finitely generated CM module over a Noetherian local ring $R$, then $\depth(p,M)=\depth_{R_{p}}M_{p}$ for all $p$ in the support of $M$.
\begin{lemma} \label{whenislocalcohocm} Let $(R,m,k)$ be a CM local ring with canonical module $\omega_{R}$ and $n=\dim{R}$, let $I$ be an ideal with $d=\height{I}$, and assume $\height{p}=d$ for every minimal prime $p$ of $I$. Assume $R$ is Gorenstein or assume $\widehat{R}$ is a domain. Then the following hold.
\begin{enumerate}
\item For every prime $q\supseteq I$ with $\height{q}-d\leq 2$, we have
$$\height{q}-d =\depth_{R_{q}}H_{I}^{d}(\omega_{R})_{q}=\depth(q,H_{I}^{d}(\omega_{R})).$$
\item If $E$ is a minimal injective resolution of $H_{I}^{d}(\omega_{R})$, then 
\begin{enumerate}
\item $\ass{E^{0}}=\{q \ | \ q\supseteq I,\height{q}=d\}$
\item $\ass{E^{1}}=\{q \ | \ q\supseteq I,\height{q}=d+1\}$
\item $\ass{E^{2}}\supseteq\{q \ | \ q\supseteq I,\height{q}=d+2\}$
\item If $R$ is a domain and if $q$ is in $\ass{E^{i}}$ for some $i\geq 2$, then $\height{q}\geq d+2$.
\end{enumerate}
\end{enumerate}
\end{lemma}

\begin{proof}
First note that since $\height{p}=d$ for every minimal prime $p$ of $I$, if $q$ is a prime containing $I$, then there exists a chain of primes $p_{0}\subsetneq\dots\subsetneq p_{\height{I}}=p\subseteq q$ for some minimal prime $p$ of $I$. Since this chain is also a chain in $R_{q}$, $\height{I_{q}}=\height{I}$, so
$$H_{I}^{d}(\omega_{R})_{q}\cong H_{I_{q}}^{\height{I_{q}}}(\omega_{R_{q}})\neq 0$$
by Lemma \ref{dimsuppequalsnminusd}.\newline

Now let $M=H_{I}^{d}(\omega_{R})$. If $\height{q}=\height{I}$, then
$$M_{q}\cong H_{I_{q}}^{\height{q}}(\omega_{R_{q}})\cong E_{R_{q}}(R_{q}/qR_{q})\neq 0$$
where the second isomorphism is by \cite[Corollary 3.5.9]{brunsherzog}, so $q$ is a minimal prime in $\supp{M}$, telling us that $q$ is in $\ass{M}$. Therefore there exist injections $R/q\rightarrow M$ and $R_{q}/qR_{q}\rightarrow M_{q}$, so $\Hom_{R}(R/q,M)\neq 0$ and $\Hom_{R_{q}}(R_{q}/qR_{q},M_{q})\neq 0$ and $\depth(q,M)=\depth_{R_{q}}M_{q}=0$. This shows $\ass{E^{0}}\supseteq \{q|\height{q}=\height{I}\}$, and Lemma \ref{dualizingcomplexcomputationlemma} shows $\ass{E^{0}}\subseteq \{q|\height{q}=\height{I}\}$.\newline

Now assume $\height{q}-\height{I}=1$. Since $\ass{E^{0}}=\{q|\height{q}=\height{I}\}$, $\Hom_{R}(R/q,E^{0})=0$, so $\depth(q,M)>0$. Also by Lemma \ref{nminusdcompositionisnonzero} $H_{q_{q}}^{\height{q}-\height{I}}(M_{q})\neq 0$, so $\depth_{R_{q}}M_{q}\leq \height{q}-\height{I}$. This gives
$$\height{q}-\height{I}=1\leq \depth(q,M)\leq \depth(q_{q},M_{q})\leq \depth_{R_{q}}M_{q}\leq \height{q}-\height{I}$$
where the second inequality is by \cite[Proposition 5.2.2]{sri} and the third inequality is by \cite[Proposition 5.2.3]{sri}. This shows $\ass{E^{1}}\supseteq \{q|\height{q}=\height{I}+1\}$, and Lemma \ref{dualizingcomplexcomputationlemma} shows $\ass{E^{1}}\subseteq \{q|\height{q}=\height{I}+1\}$.\newline

Finally assume $\height{q}-\height{I}=2$. Since $\ass{E^{0}}=\{q \ | \ q\supseteq I,\height{q}=\height{I}\}$ and since $\ass{E^{1}}=\{q|\height{q}=\height{I}+1\}$, $\Hom_{R}(R/q,E^{i})=0$ for $i=0,1$, so $\depth(q,M)>1$. Also, by Lemma \ref{nminusdcompositionisnonzero} $H_{q_{q}}^{\height{q}-\height{I}}(M_{q})\neq 0$, so $\depth_{R_{q}}M_{q}\leq \height{q}-\height{I}$. Therefore 
$$\height{q}-\height{I}=2\leq \depth(q,M)\leq \depth(q_{q},M_{q})\leq \depth_{R_{q}}M_{q}\leq \height{q}-\height{I},$$
and so $\ass{E^{2}}\supseteq\{q|\height{q}=\height{I}+2\}$.\newline

To prove the last assertion, assume $R$ is a domain and suppose $q$ is a prime containing $I$ with $\height{q}-\height{I}\leq 1$. Then $\idim_{R_{q}}M_{q}=\height{q}-\height{I}$ by Remark \ref{dimrminushtigreaterthantwojustification}, so $\ext_{R_{q}}^{i}(k(q),M_{q})=0$ for $i>\height{q}-\height{I}$, which is equivalent to $\mu^{i}(q,M)=0$ for $i>\height{q}-\height{I}$. This tells us that $E_{R}(R/q)$ is not a summand of $E^{i}$ for $i>\height{q}-\height{I}$. Therefore, if $q'$ is in $\ass{E^{i}}$ for $i\geq 2$, we have $\height{q'}-\height{I}\geq 2$.
\end{proof}

In Proposition \ref{bassnumbersequalityplusone} we show that the inequality
$$\mu^{0}(q,H_{I}^{d+1}(\omega_{R}))\leq\mu^{2}(q,H_{I}^{d}(\omega_{R}))$$
can be strict. 
\begin{lemma} \label{firsttwobassnumbersanddplusonelocalcoho} Let $(R,m,k)$ be a CM local ring with canonical module $\omega_{R}$ and $n=\dim{R}$, let $I$ be an ideal with $d=\height{I}$, and assume $\height{p}=d$ for every minimal prime $p$ of $I$. Let $E$ be a minimal injective resolution of $H_{I}^{d}(\omega_{R})$ and let $D$ be the dualizing complex. Assume $R$ is Gorenstein or assume $\widehat{R}$ is a domain. Then the following hold.
\begin{enumerate}
\item 
\[\mu^{0}(q,H_{I}^{d}(\omega_{R}))= \begin{cases} 
      1 & q\supseteq I, \height{q}=d \\
      0 & \text{otherwise} 
   \end{cases}
\]
\item 
\[\mu^{1}(q,H_{I}^{d}(\omega_{R}))= \begin{cases} 
      1 & q\supseteq I, \height{q}=d+1 \\
      0 & \text{otherwise} 
   \end{cases}
\]
\item There exists an injection $H_{I}^{d+1}(\omega_{R})\rightarrow E^{2}$. In particular, $\ass{H_{I}^{d+1}(\omega_{R})}\subseteq \ass{E^{2}}$ and
$$\mu^{0}(q,H_{I}^{d+1}(\omega_{R}))\leq\mu^{2}(q,H_{I}^{d}(\omega_{R}))$$
for all primes $q$. Moreover, if $n-d\geq 3$, then $m\notin\ass{H_{I}^{d+1}(\omega_{R})}$ if and only if $\mu^{2}(m,H_{I}^{d}(\omega_{R}))=0$, and
$$\ass{H_{I}^{d+1}(\omega_{R})}\subseteq \{q \ | \ q\supseteq I,\height{q}=d+2\}$$
if and only if 
$$\ass{E^{2}}=\{q \ | \ q\supseteq I,\height{q}=d+2\}.$$
\item Assume $R$ is a domain. If $q\in\ass{E^{2}}\backslash\ass{H_{I}^{d+1}(\omega_{R})}$ and if $\height{q}=d+2$, then $\mu^{2}(q,H_{I}^{d}(\omega_{R}))=1$.
\item If $n-d\geq 2$, then $\mu^{0}(m,H_{I}^{d+1}(\omega_{R}))=\mu^{2}(m,H_{I}^{d}(\omega_{R}))$.
\item If $H_{I}^{d+1}(\omega_{R})=0$, then $\ass{E^{2}}=\{q \ | \ q\supseteq I,\height{q}=d+2\}$.
\end{enumerate}
\end{lemma}

\begin{proof}
Let $M=H_{I}^{d}(\omega_{R})$, let $E$ be a minimal injective resolution of $M$, and let $D$ be the dualizing complex. Then we have complexes
$$0\rightarrow M\xrightarrow{i} E^{0}\xrightarrow{d_{0}} E^{1}\xrightarrow{d_{1}}E^{2}\rightarrow\dots$$
and
$$0\rightarrow M\xrightarrow{i'} \Gamma_{I}(D^{d})\xrightarrow{d_{0}'} \Gamma_{I}(D^{d+1})\xrightarrow{d_{1}'} \Gamma_{I}(D^{d+2})\rightarrow\dots.$$
Recall that 
$$\Gamma_{I}(D^{d+i})=\bigoplus_{q\supseteq I,\height{q}=\height{I}+i}E_{R}(E/q).$$

1. and 2. Now consider the following diagram with exact rows

\[\begin{tikzcd}
	0 & M & {E^{0}} & {\text{coker}(i)} & 0 \\
	\\
	\\
	0 & M & {\Gamma_{I}(D^{d})} & {\text{coker}(i')} & 0
	\arrow[from=4-1, to=4-2]
	\arrow[from=4-2, to=4-3]
	\arrow[from=4-3, to=4-4]
	\arrow[from=4-4, to=4-5]
	\arrow[from=1-1, to=1-2]
	\arrow[from=1-2, to=1-3]
	\arrow[from=1-3, to=1-4]
	\arrow[from=1-4, to=1-5]
	\arrow[from=1-2, to=4-2]
	\arrow["{f_{0}}", dashed, from=1-3, to=4-3]
	\arrow["{h_{0}}", dashed, from=1-4, to=4-4]
\end{tikzcd}\]
As in the proof of Lemma \ref{dualizingcomplexcomputationlemma}, $f_{0}$ is one to one. Since $f_{0}$ is one to one and since $\ass{E^{0}}=\{q \ | \ q\supseteq I,\height{q}=\height{I}\}$ by Lemma \ref{whenislocalcohocm}, $f_{0}$ is an isomorphism, so
\[\mu^{0}(q,H_{I}^{d}(\omega_{R}))= \begin{cases} 
      1 & q\supseteq I, \height{q}=\height{I} \\
      0 & \text{otherwise} 
   \end{cases}
\]
Also, $f_{0}$ being an isomorphism implies that $h_{0}$ is an isomorphism by the snake lemma. Now consider the following diagram with exact rows

\[\begin{tikzcd}
	0 & {\text{coker}(i)} & {E^{1}} & {\text{coker}(d_{0})} & 0 \\
	\\
	\\
	0 & {\text{coker}(i')} & {\Gamma_{I}(D^{d+1})} & {\text{coker}(d_{0}')} & 0
	\arrow[from=4-1, to=4-2]
	\arrow[from=4-2, to=4-3]
	\arrow[from=4-3, to=4-4]
	\arrow[from=4-4, to=4-5]
	\arrow[from=1-1, to=1-2]
	\arrow[from=1-2, to=1-3]
	\arrow[from=1-3, to=1-4]
	\arrow[from=1-4, to=1-5]
	\arrow["{h_{0}}", dashed, from=1-2, to=4-2]
	\arrow["{f_{1}}", dashed, from=1-3, to=4-3]
	\arrow["{h_{1}}", dashed, from=1-4, to=4-4]
\end{tikzcd}\]
Since the composition
$$\text{coker}(i)\rightarrow \text{coker}(i')\rightarrow \Gamma_{I}(D^{d+1})$$
is one to one and since $E^{1}$ is an essential extension of $\text{coker}(i)$, $f_{1}$ is one to one. Since $f_{1}$ is one to one and since $\ass{E^{1}}=\{q \ | \ q\supseteq I,\height{q}=\height{I}+1\}$ by Lemma \ref{whenislocalcohocm}, $f_{1}$ is an isomorphism, so
\[\mu^{1}(q,H_{I}^{d}(\omega_{R}))= \begin{cases} 
      1 & q\supseteq I, \height{q}=\height{I}+1 \\
      0 & \text{otherwise} 
   \end{cases}
\]
Also, since $h_{0}$ and $f_{1}$ are isomorphisms, $h_{1}$ is an isomorphism by the snake lemma.\newline

3. To see that $\ass{H_{I}^{d+1}(\omega_{R})}\subseteq \ass{E^{2}}$, notice that we have an exact sequence
$$0\rightarrow \text{ker}d_{1}'/\text{im}d_{0}'\rightarrow \Gamma_{I}(D^{d+1})/\text{im}d_{0}'\rightarrow \Gamma_{I}(D^{d+1})/\text{ker}d_{1}'\rightarrow 0.$$
Since $H_{I}^{d+1}(\omega_{R})\cong \text{ker}d_{1}'/\text{im}d_{0}'$ and since $\text{coker}(d_{0})\cong \text{coker}(d_{0}')\cong \Gamma_{I}(D^{d+1})/\text{im}d_{0}'$, the exact sequence becomes
$$\ddagger\ddagger\quad0\rightarrow H_{I}^{d+1}(\omega_{R})\rightarrow \text{coker}(d_{0})\rightarrow {im}d_{1}'\rightarrow 0,$$
giving a sequence of injections
$$H_{I}^{d+1}(\omega_{R})\rightarrow \text{coker}(d_{0})\rightarrow E^{2}.$$
We now prove the in particular part. Let $G$ be the injective hull of $H_{I}^{d+1}(\omega_{R})$. Since we have an injection $H_{I}^{d+1}(\omega_{R})\rightarrow E^{2}$ and since $G$ is an essential extension of $H_{I}^{d+1}(\omega_{R})$, there exists an injection $G\rightarrow E^{2}$, giving
$$\mu^{0}(q,H_{I}^{d+1}(\omega_{R}))\leq\mu^{2}(q,H_{I}^{d}(\omega_{R}))$$
for all primes $q$.\newline

We now prove the moreover part. The injection $H_{I}^{d+1}(\omega_{R})\rightarrow E^{2}$ shows that $m\notin\ass{E^{2}}$ implies $m\notin\ass{H_{I}^{d+1}(\omega_{R})}$. Also, $\ddagger\ddagger$ gives
\begin{align*}
(3)\quad\quad\ass{E^{2}} = \ass{\text{coker}(d_{0})} &\subseteq\ass{H_{I}^{d+1}(\omega_{R})}\bigcup\ass{\text{im}d_{1}'} \\
&\subseteq\ass{H_{I}^{d+1}(\omega_{R})}\bigcup\ass{\Gamma_{I}(D^{d+2})},
\end{align*}
so $m\notin\ass{H_{I}^{d+1}(\omega_{R})}$ implies $m\notin\ass{E^{2}}$ (recall that in the moreover part we are assuming $n-d\geq 3$, so $m\notin\ass{\Gamma_{R}(D^{d+2})}$).\newline

Finally,
$$\ass{H_{I}^{d+1}(\omega_{R})}\subseteq \{q \ | \ q\supseteq I,\height{q}=\height{I}+2\}$$
if and only if 
$$\ass{E^{2}}=\{q \ | \ q\supseteq I,\height{q}=\height{I}+2\}$$
by $(3)$ and the fact that $\ass{E^{2}}\supseteq\{q \ | \ q\supseteq I,\height{q}=\height{I}+2\}$ (Lemma \ref{whenislocalcohocm}).
\newline

4. Since $q$ is in $\ass{E^{2}}\backslash\ass{H_{I}^{d+1}(\omega_{R})}$, $q$ is not in $\supp{H_{I}^{d+1}(\omega_{R})}$. This is because if $q$ is in $\supp{H_{I}^{d+1}(\omega_{R})}$, then $q$ is in $\supp{E^{2}}$ (since $H_{I}^{d+1}(\omega_{R})$ is a submodule of $E^{2}$), so $\height{q}\geq\height{I}+2$ by Lemma \ref{whenislocalcohocm}. By assumption $\height{q}=\height{I}+2$, so $q$ is a minimal prime, and so in $\ass{H_{I}^{d+1}(\omega_{R})}$, a contradiction.\newline

Now localizing $\ddagger\ddagger$ at $q$ shows $(\text{coker}(d_{0}))_{q}$ and $(\text{im}d_{1}')_{q}$ are isomorphic. The injection
$$\text{im}d_{1}'\rightarrow \Gamma_{I}(D^{d+2})$$
gives an injection
$$(\text{im}d_{1}')_{q}\rightarrow E_{q}(R/q),$$
so $E_{R}(R/q)$ is the injective hull of $(\text{im}d_{1}')_{q}$ since $E_{R}(R/q)$ is indecomposable. Since $E_{q}^{2}$ is the injective hull of $(\text{coker}(d_{0}))_{q}$ and since $(\text{coker}(d_{0}))_{q}\cong (\text{im}d_{1}')_{q}$, we have $E_{q}^{2}\cong E_{q}(R/q)$, proving that $\mu^{2}(q,H_{I}^{d}(\omega_{R}))=1$.\newline

5. Applying $\Hom_{R}(k,\_)$ to $\ddagger\ddagger$ we get
$$\Hom_{R}(k,H_{I}^{d+1}(\omega_{R}))\cong\Hom_{R}(k,\text{coker}d_{0}),$$
so $\mu^{0}(m,H_{I}^{d+1}(R))=\mu^{0}(m,\text{coker}d_{0})=\mu^{2}(m,H_{I}^{d}(R))$.\newline

6. If $H_{I}^{d+1}(\omega_{R})=0$, then $(3)$ implies 
$$\ass{H_{I}^{d}(\omega_{R})}\subseteq \{q \ | \ q\supseteq I,\height{q}=\height{I}+2\}.$$
Also, $\ass{H_{I}^{d}(\omega_{R})}\supseteq \{q \ | \ q\supseteq I,\height{q}=\height{I}+2\}$ by Lemma \ref{whenislocalcohocm}.
\end{proof}

\begin{lemma} \label{mnotinassofsecondnonzerolocalcohoequivalentstatments} Let $(R,m,k)$ be a CM local domain with canonical module $\omega_{R}$ and $n=\dim{R}$, let $I$ be an ideal with $d=\height{I}$, and assume $\height{p}=d$ for every minimal prime $p$ of $I$. Assume $R$ is Gorenstein or assume $\widehat{R}$ is a domain. If $E$ is the minimal injective resolution of $H_{I}^{d}(\omega_{R})$ and $n-d=3$, then the following are equivalent
\begin{enumerate}
\item $m\notin\ass{H_{I}^{d+1}(\omega_{R})}$
\item $\mu^{2}(m,H_{I}^{d}(\omega_{R}))=0$
\item $\ass{E^{2}}=\{q \ | \ q\supseteq I,\height{q}=d+2\}$
\end{enumerate}
\end{lemma}

\begin{proof}
Note that 1 and 2 are equivalent by Lemma \ref{firsttwobassnumbersanddplusonelocalcoho}, so we just need to show 2 and 3 are equivalent. Notice $\mu^{2}(m,H_{I}^{d}(\omega_{R}))=0$ if and only if $m\not\in\ass{E^{2}}$. Also, by Lemma \ref{whenislocalcohocm} 
$$\ass{E^{2}}\supseteq\{q \ | \ q\supseteq I,\height{q}=\height{I}+2\}$$
and $q\in\ass{E^{2}}$ implies $\height{q}\geq\height{I}+2$. Therefore, if $q$ is in $\ass{E^{2}}$, then $\height{q}=\height{I}+2$ or $\height{q}=\height{I}+3$, so $m\not\in\ass{E^{2}}$ is equivalent to
$$\ass{E^{2}}=\{q \ | \ q\supseteq I,\height{q}=\height{I}+2\}.$$
\end{proof}

Note that if $n-d=0$, then $\sqrt{I}=m$, so $H_{I}^{d}(\omega_{R})\cong E_{R}(k)$ by \cite[Corollary 3.5.9]{brunsherzog}. This justifies the assumption $d<n$ in the following lemma.
\begin{lemma} \label{upperboundfordimensionofsupportoflocalcohomology} Let $(R,m,k)$ be a Noetherian local domain with $n=\dim{R}$ and let $I$ be an ideal with $d=\height{I}$. Assume $d<n$. If $\height{p}=d$ for every minimal prime $p$ of $I$, then
$$\dim{\supp{H_{I}^{d+i}(M)}}\leq n-(d+i+1)$$
for every module $M$ and all $i>0$.
\end{lemma}

\begin{proof}
Consider $H_{I}^{d+i}(M)$ for some $i>0$ and let $q$ be a prime containing $I$ with $\height{q}=\height{I}+i$. Then $\height{I}=\height{I_{q}}$ since if $p_{0}\subseteq\dots\subseteq p_{d}=p\subseteq q$ is a chain in $R$, then it is a chain in $R_{q}$. This tells us that 
$$\height{I_{q}}=\height{I}<\height{q}=\dim{R_{q}},$$
so $\sqrt{0+I_{q}}\neq q_{q}$, and $H_{I_{q}}^{d+i}(R_{q})=0$ by \cite[Theorem 1.1]{bagheriyeh}. Since the cohomological dimension of $N$ is less than or equal to the cohomological dimension of $R_{q}$ for any $R_{q}$ module $N$, we have 
$$H_{I}^{d+i}(M)_{q}\cong H_{I_{q}}^{d+i}(M_{q})=0,$$
and so $q$ is not in $\supp{H_{I}^{d+i}(M)}$ for any prime $q$ with $\height{q}=\height{I}+i$. Therefore
$$\dim{\supp{H_{I}^{d+i}(M)}}<\dim{R}-\height{q}=\dim{R}-(\height{I}+i).$$
\end{proof}

\begin{definition} Let $M$ be a module over a ring $R$ and let $I$ be an ideal. We say $M$ is $I$ cofinite if $\supp{M}\subseteq V(I)$ and $\ext_{R}^{i}(R/I,M)$ is finitely generated for all $i$.
\end{definition}

\begin{remark} \label{injectivedimensionofdplusonelocalcohoupperbound} Let $(R,m,k)$ be a regular local ring of dimension $n$ containing  a field and let $I$ be an ideal with $d=\height{I}$. Assume $d<n$. If $\height{p}=d$ for every minimal prime $p$ of $I$, then
$$\idim_{R}(H_{I}^{d+1}(R))\leq \dim{\supp{H_{I}^{d+1}(R)}}\leq n-d-2.$$
\end{remark}

\begin{proof}
The first inequality is by \cite[Corollary 3.9]{hunekeone} and \cite[Theorem 3.4.a]{lyubeznik} and the second inequality is by Lemma \ref{upperboundfordimensionofsupportoflocalcohomology}.
\end{proof}

Note that Remark \ref{injectivedimensionofdplusonelocalcohoupperbound} shows if $d<n$ and if $\idim_{R}H_{I}^{d+1}(R)=n-d-2$, then $\dim{\supp{H_{I}^{d+1}(R)}}=n-d-2$. Also, the only part of the following proposition that needs $H_{I}^{d+1}(R)$ to be $I$ cofinite is the proof of 1 implies 4.
\begin{proposition} \label{nminusdequalsthreewithfourequivalentconditions} Let $(R,m,k)$ be a regular local ring of dimension $n$ containing a field, let $I$ be an ideal with $d=\height{I}$, and assume $\height{p}=d$ for every minimal prime $p$ of $I$. Let $E$ be the minimal injective resolution of $H_{I}^{d}(R)$ and assume $n-d=3$. If $H_{I}^{d+1}(R)\neq 0$ and if $H_{I}^{d+1}(R)$ is $I$ cofinite, then the following are equivalent
\begin{enumerate}
\item $m\notin\ass{H_{I}^{d+1}(R)}$
\item $\mu^{2}(m,H_{I}^{d}(R))=0$
\item $\ass{E^{2}}=\{q \ | \ q\supseteq I,\height{q}=d+2\}$
\item $\idim_{R}H_{I}^{d+1}(R)=1$ and every associated prime of $H_{I}^{d+1}(R)$ has the same height
\end{enumerate}
\end{proposition}

\begin{proof}
Lemma \ref{mnotinassofsecondnonzerolocalcohoequivalentstatments} shows $1,2,3$ are equivalent, so it is enough to show $1$ and $4$ are equivalent. First note that 
$$(4)\quad\quad\idim_{R}(H_{I}^{d+1}(R))\leq \dim{\supp{H_{I}^{d+1}(R)}}\leq n-d-2=1$$
by Remark \ref{injectivedimensionofdplusonelocalcohoupperbound}.\newline

If 1 holds, then $0<\dim{\supp{H_{I}^{d+1}(R)}}$ since $H_{I}^{d+1}(R)\neq 0$ and since $m\notin\ass{H_{I}^{d+1}(R)}$. Also, since $H_{I}^{d+1}(R)$ is $I$ cofinite, \cite[Theorem 2.3]{hellus} gives
$$\dim{\supp{H_{I}^{d+1}(R)}}\leq \idim_{R}(H_{I}^{d+1}(R)),$$
so $(4)$ implies $\idim_{R}(H_{I}^{d+1}(R))=1$. Now since $n-d=3$, since $\ass{H_{I}^{d+1}(R)}\subseteq\ass{E^{2}}$ (Lemma \ref{firsttwobassnumbersanddplusonelocalcoho}), and since $q\in\ass{E^{2}}$ implies $\height{q}\geq\height{I}+2$ (Lemma \ref{whenislocalcohocm}), $q\in\ass{H_{I}^{d+1}(R)}$ implies $\height{q}=\height{I}+2$ or $\height{q}=\height{I}+3$. Since $m\notin\ass{H_{I}^{d+1}(\omega_{R})}$, $q\in\ass{H_{I}^{d+1}(R)}$ implies $\height{q}=\height{I}+2$, proving 4.\newline

Now assume 4. Then $(4)$ and the assumption $\idim_{R}H_{I}^{d+1}(R)=1$ give
$$1=\idim_{R}(H_{I}^{d+1}(R))\leq \dim{\supp{H_{I}^{d+1}(R)}}\leq n-d-2=1,$$
so $\dim{\supp{H_{I}^{d+1}(R)}}=1$. Therefore there exists a minimal prime $q$ of $H_{I}^{d+1}(R)$ with $q\supseteq I$ and $\height{q}=\height{I}+2$. Since $q$ is an associated prime and since every associated prime of $H_{I}^{d+1}(R)$ has the same height, we have $m\notin\ass{H_{I}^{d+1}(R)}$, proving 1.
\end{proof}

Recall that if $R$ is a regular local ring containing a field, then $\mu^{i}(q,H_{I}^{j}(R))<\infty$ for all $i$, all $j$, all $I$, and all primes $q$ by \cite[Theorem 2.1]{hunekeone} and \cite[Theorem 3.4.d]{lyubeznik}. Also, regular local rings are Gorenstein and have the property that $\widehat{R}$ is a domain. Note that we have been assuming either $R$ is Gorenstein or $\widehat{R}$ is a domain in many of our results.
\begin{proposition} \label{bassnumbersequalityplusone} Let $(R,m,k)$ be a regular local ring containing a field with $n=\dim{R}$, let $I$ be an ideal with $d=\height{I}$, and assume $\height{p}=d$ for every minimal prime $p$ of $I$. If $n-d\geq 3$, then
$$\mu^{0}(q,H_{I}^{d+1}(R))+1=\mu^{2}(q,H_{I}^{d}(R))$$
for every $q\in\ass{H_{I}^{d+1}(R)}$ with $\height{q}=d+2$.
\end{proposition}

\begin{proof}
First note that since $\height{p}=d$ for every minimal prime $p$ of $I$, if $q$ is a prime containing $I$, then there exists a chain of primes $p_{0}\subsetneq\dots\subsetneq p_{\height{I}}=p\subseteq q$. Since this chain is also a chain in $R_{q}$, $\height{I_{q}}=\height{I}$. Now let 
$$E=0\rightarrow E^{0}\xrightarrow{d_{0}} E^{1}\xrightarrow{d_{1}}\dots$$
be the minimal injective resolution of $H_{I}^{d}(R)$, let $D$ be the dualizing complex, and let
$$\Gamma_{I}(D)=0\rightarrow \Gamma_{I}(D^{d})\xrightarrow{d_{0}'} \Gamma_{I}(D^{d+1})\xrightarrow{d_{1}'} \dots.$$
Then the proof of Lemma \ref{firsttwobassnumbersanddplusonelocalcoho} gives the exact sequence
$$\ddagger\ddagger\quad0\rightarrow H_{I}^{d+1}(R)\rightarrow \text{coker}(d_{0})\rightarrow {im}d_{1}'\rightarrow 0.$$
Since $\height{I_{q}}=\height{I}$, we have
$$\idim_{R_{q}}H_{I}^{d}(R)_{q}=\dim{R_{q}}-\height{I_{q}}=\height{q}-\height{I}=2$$
where the first equality is by \cite[Proposition 3.5]{dorreh}, so
$$E_{q}=0\rightarrow (E^{0})_{q}\xrightarrow{(d_{0})_{q}} (E^{1})_{q}\xrightarrow{(d_{1})_{q}} (E^{2})_{q}\rightarrow 0.$$
This tells us that 
$$(\text{coker}(d_{0}))_{q}\cong (E^{2})_{q}\cong\bigoplus E_{R}(R/q),$$
which is a finite sum. Also,
$$\Gamma_{I}(D)_{q}=0\rightarrow \Gamma_{I}(D^{d})_{q}\xrightarrow{(d_{0}')_{q}} \Gamma_{I}(D^{d+1})_{q}\xrightarrow{(d_{1}')_{q}} \Gamma_{I}(D^{d+2})_{q}\rightarrow 0.$$
Since
$$\height{I_{q}}=\height{I}<\height{I}+2=\height{q}=\dim{R_{q}},$$
$H_{I}^{d+2}(R)_{q}=0$ by \cite[Theorem 1.1]{bagheriyeh}, giving the first equality below
$$(\text{im}d_{1}')_{q}=(\text{ker}d_{2}')_{q}=\Gamma_{I}(D^{d+2})_{q}=E_{R}(R/q).$$
Finally,
$$\idim_{R_{q}}H_{I}^{d+1}(R)_{q}\leq \dim{\supp{(H_{I}^{d+1}(R)_{q}}}\leq \height{q}-(\height{I}+1+1) =\height{q}-\height{I}-2=0,$$
where the first inequality is by \cite[Corollary 3.9]{hunekeone} and \cite[Theorem 3.4.a]{lyubeznik} and the second inequality is by Lemma \ref{upperboundfordimensionofsupportoflocalcohomology}, so $H_{I}^{d+1}(R)_{q}=\oplus E_{R}(R/q)$ (which is a finite sum). Localizing $\ddagger\ddagger$ at $q$ we get
$$0\rightarrow \bigoplus E_{R}(R/q)\rightarrow \bigoplus E_{R}(R/q)\rightarrow E_{R}(R/q)\rightarrow 0,$$
finishing the proof.
\end{proof}

\section{Examples} \label{examples}
The goal of this section is to give some examples of local cohomology modules that satisfy some assumptions and conclusions that have shown up throughout the paper. The following follows from the fact that $\omega_{R}$ is resolved by the dualizing complex.

\begin{observation} Let $(R,m,k)$ be a CM local ring with canonical module $\omega_{R}$ and $n=\dim{R}$, let $I$ be an ideal with $d=\height{I}$, and let $D$ be the dualizing complex. If $H_{I}^{d}(\omega_{R})$ is the only nonzero local cohomology module, then $\Gamma_{I}(D)$ is an injective resolution of $H_{I}^{d}(\omega_{R})$. In particular, $\mu^{i}(m,H_{I}^{d}(\omega_{R}))=0$ for $i<n-d$.
\end{observation}

The proof of the following was given to us by Mark Walker. 
\begin{lemma} \label{markspectralsequence} Let $R$ be a Noetherian local ring, $I\subseteq J$ be ideals, and let $M$ be a module. If there exists $j$ so that $H_{I}^{i}(M)=0$ for $i\neq j,j+1$, then there is an exact sequence
\begin{align*}
\dots&\rightarrow\ext_{R}^{t-j}(R/J,H_{I}^{j}(M)) \rightarrow\ext_{R}^{t}(R/J,M)\rightarrow\ext_{R}^{t-j-1}(R/J,H_{I}^{j+1}(M)) \\
&\rightarrow\ext_{R}^{t-j+1}(R/J,H_{I}^{j}(M))\rightarrow\dots.
\end{align*}
\end{lemma}

\begin{proof}
The proof of \cite[Proposition 3.1]{huneke} gives a spectral sequence
$$E_{2}^{p,q}=\ext_{R}^{p}(R/J,H_{I}^{q}(M))\implies \ext_{R}^{p+q}(R/J,M).$$
The lemma now follows from the assumption $H_{I}^{i}(M)=0$ for $i\neq j,j+1$.
\end{proof}

The proof of the following is very similar to the proof of \cite[Proposition 3.1]{huneke}.
\begin{lemma} \label{dthlocalcohoicofiniteimpliesdplusonelocalcohocofinite} Let $(R,m,k)$ be a CM local ring with canonical module $\omega_{R}$ and $n=\dim{R}$, let $I\subseteq J$ be ideals with $R/J$ CM and $d=\height{I}$, and assume $H_{I}^{i}(\omega_{R})=0$ for $i\neq d,d+1$. Then
$$\ext_{R}^{i-2}(R/J,H_{I}^{d+1}(\omega_{R}))\cong\ext_{R}^{i}(R/J,H_{I}^{d}(\omega_{R}))$$
for $i<\height{J}-d$ and $i>\height{J}-d+1$. In particular, if $J$ is prime, then
$$\mu^{i-2}(J,H_{I}^{d+1}(\omega_{R}))=\mu^{i}(J,H_{I}^{d}(\omega_{R}))$$
for $i<\height{J}-d$ and $i>\height{J}-d+1$.
\end{lemma}

\begin{proof}
Lemma \ref{markspectralsequence} gives a long exact sequence
\begin{align*}
\dots&\rightarrow\ext_{R}^{t-d}(R/J,H_{I}^{d}(\omega_{R})) \rightarrow\ext_{R}^{t}(R/J,\omega_{R})\rightarrow\ext_{R}^{t-d-1}(R/J,H_{I}^{d+1}(\omega_{R})) \\
&\rightarrow\ext_{R}^{t-d+1}(R/J,H_{I}^{d}(\omega_{R}))\rightarrow\dots.
\end{align*}
Note that 
$$\height{J}=\dim{R}-\dim{R/J}=\depth{R}-\depth{R/J}$$
(since $R/J$ is CM) and $\ext_{R}^{i}(R/J,\omega_{R})=0$ for $i<\depth{R}-\dim{R/J}$ by \cite[Theorem 17.1]{matsumura}. Also,
$$\sup\{i|\ext_{R}^{i}(R/J,\omega_{R})\neq 0\}=\depth{R}-\depth{R/J}$$
by \cite[Exercise 3.1.24]{brunsherzog}, so $\ext_{R}^{i}(R/J,\omega_{R})=0$ for $i>\depth{R}-\depth{R/J}$. Therefore $\ext_{R}^{i}(R/J,\omega_{R})\neq 0$ for $\height{J}$, and the result follows.
\end{proof}

Recall that in Lemma \ref{firsttwobassnumbersanddplusonelocalcoho} we showed that if $(R,m,k)$ is a CM local ring with canonical module $\omega_{R}$ and $n=\dim{R}$, if $I$ is an ideal with $d=\height{I}$, if $\height{p}=d$ for every minimal prime $p$ of $I$, and if $n-d\geq 3$, then $m\in\ass{H_{I}^{d+1}(\omega_{R})}$ if and only if $\mu^{2}(m,H_{I}^{d}(\omega_{R}))\neq 0$. The following example gives a case where $H_{I}^{d+1}(\omega_{R})\neq 0$.
\begin{example} Let $k$ be a field of characteristic zero, let $X$ be a $2\times 3$ matrix of indeterminates, let $R=k[X]$, let $m$ be the maximal homogeneous ideal, and let $I$ be the ideal generated by the maximal minors of $X$. Then $H_{I}^{2}(R)$ and $H_{I}^{3}(R)$ are the only two nonzero local cohomology modules by \cite[Theorem 1.1]{witt}. This tells us that $\height{I}=2$, so $\dim{R}-\height{I}=6-2=4$. Also, \cite[Theorem 1.1]{witt} gives $H_{I}^{3}(R)\cong E_{R}(k)$, so $m\in\ass{H_{I}^{3}(R)}$ and $\mu^{0}(m,H_{I}^{3}(R))\neq 0$. Since
$$2<4=\dim{R}-\height{I}=\height{m}-\height{I},$$
Lemma \ref{dthlocalcohoicofiniteimpliesdplusonelocalcohocofinite} gives the equality below
$$\mu^{2}(m,H_{I}^{\height{I}}(R))=\mu^{0}(m,H_{I}^{\height{I}+1}(R))\neq 0.$$
\end{example}

Recall that in Lemma \ref{firsttwobassnumbersanddplusonelocalcoho} we showed that if $(R,m,k)$ is a CM local ring with canonical module $\omega_{R}$ and $n=\dim{R}$, if $I$ is an ideal with $d=\height{I}$, if $\height{p}=d$ for every minimal prime $p$ of $I$, and if $H_{I}^{d+1}(\omega_{R})=0$, then $\ass{E^{2}}=\{q \ | \ q\supseteq I,\height{q}=d+2\}$. Lemma \ref{firsttwobassnumbersanddplusonelocalcoho} also shows that for $i=0,1$ we have
\[\mu^{i}(q,H_{I}^{d}(\omega_{R}))= \begin{cases} 
      1 & q\supseteq I, \height{q}=d+i \\
      0 & \text{otherwise} 
   \end{cases}
\]
The following example gives a case where $H_{I}^{d+1}(\omega_{R})=0$.
\begin{example} Let $k$ be a field of characteristic zero, let $X$ be a $2\times 7$ matrix of indeterminates, let $R=k[X]$, let $m$ be the maximal homogeneous ideal, and let $I$ be the ideal generated by the maximal minors of $X$. Then $H_{I}^{6}(R)$ and $H_{I}^{11}(R)$ are the only two nonzero local cohomology modules by \cite[Theorem 1.1]{witt}. This tells us that $\height{I}=6$ and $H_{I}^{\height{I}+1}(\omega_{R})=0$. Also, $I$ is prime by \cite[Theorem 1]{hochstereagon}. Now localizing at $m$, we see that
$$0\rightarrow \Gamma_{I}(D^{d})\rightarrow \Gamma_{I}(D^{d+1})\rightarrow \Gamma_{I}(D^{d+2})\rightarrow \Gamma_{I}(D^{d+3})$$
is the beginning of an injective resolution of $H_{I}^{d}(R)_{m}$, where $D$ is the dualizing complex of $R_{m}$ and $d=\height{I}$. This gives
\[\mu^{i}(q,H_{I}^{d}(R))= \begin{cases} 
      1 & m\supseteq q\supseteq I, \height{q}=d+i \\
      0 & \text{otherwise} 
   \end{cases}
\]
for $i=0,1,2$.
\end{example}

The following gives an example where $m\notin\ass{H_{I}^{\height{I}+1}(R)}$ and $\dim{R}-\height{I}=3$.
\begin{example} Let $k$ be a field, let $R=k[u,v,w,x,y,z]$, let $J_{1}=(u,v,w)$, let $J_{2}=(x,y,z)$, and let $I=J_{1}\cap J_{2}$. Consider the Mayer Vietoris exact sequence
\[\begin{tikzcd}
	0 && {H_{J_{1}+J_{2}}^{0}(R)} && {H_{J_{1}}^{0}(R)\bigoplus H_{J_{2}}^{0}(R)} && {H_{J_{1}\cap J_{2}}^{0}(R)} \\
	\\
	{} && {H_{J_{1}+J_{2}}^{1}(R)} && {H_{J_{1}}^{1}(R)\bigoplus H_{J_{2}}^{1}(R)} && {H_{J_{1}\cap J_{2}}^{1}(R)} && \cdots.
	\arrow[from=1-1, to=1-3]
	\arrow[from=1-3, to=1-5]
	\arrow[from=1-5, to=1-7]
	\arrow[from=3-1, to=3-3]
	\arrow[from=3-3, to=3-5]
	\arrow[from=3-5, to=3-7]
	\arrow[from=3-7, to=3-9]
\end{tikzcd}\]
Since $J_{1}+J_{2}=(u,v,w,x,y,z)$, $H_{J_{1}+J_{2}}^{i}(R)=0$ for $i\neq \dim{R}$. Also, both $J_{1}$ and $J_{2}$ are generated by three elements, so $H_{J_{1}}^{i}(R)=H_{J_{2}}^{i}(R)=0$ for $i>3$. Putting these two facts together and using the Mayer Vietoris exact sequence, we get $H_{I}^{4}(R)=0$, so $m\notin\ass{H_{I}^{4}(R)}$.
\end{example}

\begin{definition} A Noetherian ring $R$ satisfies Serre's condition $(S_{l})$ if $\depth{R_{p}}\geq\min\{l,\dim{R_{p}}\}$ for all primes $p$.
\end{definition}

The following gives an example where $\mu^{2}(m,H_{I}^{\height{I}}(R))\neq 0$ and $\dim{R}-\height{I}=3$.
\begin{example} \label{alexandraexample} Let $R=\mathbb{Q}[x_{1},x_{2},x_{3},x_{4},x_{5},x_{6}]$, $m=(x_{1},x_{2},x_{3},x_{4},x_{5},x_{6})$, and consider the simplicial complex
$$\{4,5,6\}, \{3,5,6\}, \{2,3,5\}, \{2,3,4\}, \{1,3,4\}, \{2,4,6\}.$$
Then the ideal corresponding to this simplicial complex is 
$$I=(x_{1}x_{2},x_{1}x_{5},x_{2}x_{4}x_{5},x_{3}x_{4}x_{5},x_{1}x_{6},x_{2}x_{3}x_{6},x_{3}x_{4}x_{6},x_{2}x_{5}x_{6})\subseteq R.$$
Then $R/I$ is $(S_{2})$ by \cite[Example 39]{holmes} and all the associated primes of $R/I$ have the same height, but $R/I$ is not CM since $\dim{R/I}=3$ and $\depth{R/I}=2$. The following equality is by \cite[Theorem 1.1]{yanagawa}
$$\mu^{2}(m,H_{I}^{3}(R))=\dim_{\mathbb{Q}}[\ext_{R}^{4}(\ext_{R}^{3}(R/I,R))]_{0},$$
and by a computation in Macaulay2 we get
$$\dim_{\mathbb{Q}}[\ext_{R}^{4}(\ext_{R}^{3}(R/I,R))]_{0}=1>0.$$
\end{example}

The following generalizes Example \ref{alexandraexample}.
\begin{proposition} \label{alexandraproposition} Let $R=k[x_{1},\dots,x_{n}]$ where $k$ is a field and let $I$ be a square free monomial ideal with $d=\height{I}$, $R/I$ $(S_{2})$, and $n-d=3$. Assume all the minimal primes of $I$ have the same height. If $R/I$ is not CM, then $\mu^{2}(m,H_{I}^{d}(R))\neq 0$ where $m=(x_{1},\dots,x_{n})$. 
\end{proposition}

\begin{proof} 
Let $q\supseteq I$ be a prime ideal, let $M=H_{I}^{d}(R)$, and let 
$$E=0\rightarrow E^{0}\rightarrow E^{1}\rightarrow E^{2}\rightarrow E^{3}\rightarrow 0$$
be a minimal injective resolution of $M$. Note that $d=\height{I_{q}}$ since every minimal prime of $I$ has the same height. We start by proving four claims.\newline

\underline{\textbf{Claim 1}} If $\mu^{i}(q,M)\neq 0$ for $i=0$ or $i=1$, then $\height{q}=d+i$.\newline
Let $m\supseteq q$ be a maximal ideal. Then $M_{m}\cong H_{I_{m}}^{\height{I_{m}}}(R_{m})$. The claim now follows from Lemma \ref{whenislocalcohocm} since 
$$\ass{E^{i}}=\{q \ | \ q\supseteq I,\height{q}=d+i\}$$
for $i=0,1$.

\underline{\textbf{Claim 2}} If $\height{q}<n$, then $\mu^{n-d}(q,M)=0$.\newline
Since $M_{q}\cong H_{I_{q}}^{\height{I_{q}}}(R_{q})$,
$$\idim_{R_{q}}M_{q}=\height{q}-d$$
by \cite[Proposition 3.5]{dorreh}, so $\mu^{i}(q,M)=0$ for $i>\height{q}-d$.

\underline{\textbf{Claim 3}} $\mu^{n-d}(m,M)\neq 0$ for all maximal ideals $m$ containing $I$.\newline
If $m$ is a maximal ideal, then $M_{q}\cong H_{I_{m}}^{\height{I_{m}}}(R_{m})$, so
$$\idim_{R_{m}}M_{m}=\height{m}-d$$
by \cite[Proposition 3.5]{dorreh}. By claim 2 $\mu^{n-d}(q,M)=0$ if $\height{q}<n$, so we must have $\mu^{n-d}(m,M)\neq 0$.

\underline{\textbf{Claim 4}} If $i=0,1,2$, then $\mu^{j}(q,M)=0$ for $j\neq \height{q}-d+i$.\newline
This is an immediate consequence of Claim 1 and Claim 2.\newline

The four claims have now been proved. By the discussion before Corollary 3.16 in \cite{yanagawa}, $R/I$ is CM if and only if $R/I$ is $(S_{2})$ and $\ext_{R}^{d}(R/I,\omega_{R})$ is CM. By assumption $R/I$ is $(S_{2})$ but not CM, so $\ext_{R}^{d}(R/I,\omega_{R})$ is not CM. Also, since $R/I$ is $(S_{2})$, $R/I$ has pure dimension by \cite[Lemma 2.6]{murai}, so \cite[Corollary 3.18]{yanagawa} can be applied. Since $\ext_{R}^{d}(R/I,\omega_{R})$ is not CM, the following property fails by \cite[Corollary 3.18]{yanagawa}:
\begin{quote}
$\mu^{i}(p,H_{I}^{d}(R))=0$ for all $i$ and all prime ideals $p$ with $\dim{R/p}\neq n-d-i$.
\end{quote}
Since $n-\height{p}=\dim{R/p}$, the following property fails:
\begin{quote}
$\mu^{i}(p,H_{I}^{d}(R))=0$ for all $i$ and all prime ideals $p$ with $\height{p}\neq d+i$.
\end{quote}
By the proof of \cite[Corollary 3.18]{yanagawa} and by \cite[Theorem 1.2.3]{goto}, the property above is equivalent to the following property:
\begin{quote}
$\mu^{i}(p,H_{I}^{d}(R))=0$ for all $i$ and all monomial prime ideals $p$ with $\height{p}\neq d+i$.
\end{quote}
This tells us that $\mu^{i}(p,H_{I}^{d}(R))\neq 0$ for some monomial prime ideal $p$ with $\height{p}\neq d+i$. Using the claims, we now see that $\mu^{2}(m,H_{I}^{d}(R))\neq 0$ where $m=(x_{1},\dots,x_{n})$.
\end{proof}

\end{document}